\documentclass[11pt]{article}

\usepackage[T1]{fontenc}
\usepackage[utf8]{inputenc}
\usepackage[a4paper,margin=1in]{geometry}
\usepackage{amsmath,amssymb,amsthm}
\usepackage{xcolor}
\usepackage{hyperref}

\hypersetup{
  colorlinks=true,
  linkcolor=blue,
  citecolor=blue,
  urlcolor=blue
}

\newtheorem{thm}{Theorem}[section]
\newtheorem{cor}[thm]{Corollary}

\newtheorem{prop}[thm]{Proposition}
\theoremstyle{definition}
\newtheorem{defn}[thm]{Definition}
\theoremstyle{remark}
\newtheorem{rem}[thm]{Remark}
\newtheorem*{ex}{Example}
\numberwithin{equation}{section}

\newcommand{\dom}{\operatorname{dom}}
\newcommand{\ran}{\operatorname{ran}}
\newcommand{\mul}{\operatorname{mul}}
\newcommand{\dist}{\operatorname{dist}}
\DeclareMathOperator{\cl}{cl}

\title{On the Numerical Range of Linear Relations in Banach Spaces}
\author{
W.~Boubaker\thanks{Faculty of Sciences of Sfax, Department of Mathematics, BP1171, Sfax 3000, Tunisia. Email: \texttt{wissalboubaker1@gmail.com}}
\and
H.~Gernandt\thanks{Bergische Universit\"at Wuppertal, Gau\ss stra\ss e 20, 42119 Wuppertal, Germany. Email: \texttt{gernandt@uni-wuppertal.de}}
\and
W.~Selmi\thanks{Faculty of Sciences of Sfax, Department of Mathematics, BP1171, Sfax 3000, Tunisia. Email: \texttt{selmiwafa91@yahoo.fr}}
}
\date{}

\begin{document}
\maketitle

\begin{abstract}
This paper is devoted to the study of the numerical range of linear relations in Banach spaces. We present a new definition adapted to the multivalued nature of linear relations and analyze its main properties. We establish spectral inclusion results showing that the spectrum is contained in the union of the closure of the numerical range of a~linear relation and the numerical range of its Banach adjoint, together with resolvent estimates related to the distance to the numerical range. As an application, we derive spectral enclosures for operator pencils by associating them with suitable linear relations and introduce corresponding numerical ranges for operator pencils in Banach spaces. We show that this approach may provide sharper information than classical numerical ranges of operator pencils. Furthermore, we use the numerical abscissas to show that the Banach space numerical range can yield strict exponential decay for semigroups that cannot be obtained from the corresponding Hilbert space numerical range.
\end{abstract}



\section{Introduction}
\hskip0.5cm  The numerical range is a fundamental tool in operator theory and spectral analysis. The study of the numerical range of an operator and its generalizations has a long and  distinguished history. In 1918, O. Toeplitz \cite{ot} and F. Hausdorff \cite{ha} introduced the notion of the numerical range of a matrix  $A$ acting on a finite-dimensional space. 

For a bounded linear operator acting on a Hilbert space, the numerical range provides important information about spectral properties and resolvent estimates, see \cite{Gil} and \cite{Hilb}. In particular, it is well known that the spectrum of an operator is contained in the closure of its numerical range and that the resolvent satisfies estimates involving the distance to this set. Classical references for these results include the works of Bonsall and Duncan \cite{BD1,BD} and Kato \cite{Kato95}.

The concept of numerical range has also been extended to operators acting on Banach spaces by using the dual space  \cite{Bauer1962,Jahedi2012,Lumer61}. Recent developments also highlight the role of the numerical range as a spectral set in Banach algebras, see \cite{Blazhko_Homza_Schwenninger_deVries_Wojtylak_2025}.

Many problems arising in spectral theory naturally lead
to the study of \emph{linear relations}, which are multivalued generalizations of linear operators. These structures also play an important role in the analysis of dissipative systems and Hamiltonian problems; see, for instance, \cite{mehl2025spectral}.

Linear relations appear in several contexts such as differential operators, operator pencils, and extensions of symmetric operators. Their spectral theory has been developed in various works; see, for instance, Cross \cite{Cross98}  Ammar and Jeribi \cite{AmmarJeribi2021} and the references therein.

Recently, the numerical range of linear relations in Hilbert spaces was studied in \cite{Ammar2025}. The purpose of the present paper is to study the numerical range of linear relations in Banach spaces and to establish several properties that extend classical results known for operators. We introduce a definition of the numerical range adapted to the multivalued nature of linear relations and investigate its fundamental properties.

We prove that the point spectrum is contained in the numerical range, that the approximate point spectrum is contained in its closure, and that the full spectrum is contained in the union of the closure of the numerical range and the numerical range of the Banach adjoint relation (see Theorem~\ref{theorem1}). In particular, if $T$ is bounded and defined on the whole space, then we obtain that the spectrum is contained in the closure of the numerical range. These results generalize classical statements known for bounded operators. We also provide conditions under which the numerical range coincides with the entire complex plane and discuss situations where it is bounded (see Theorem~\ref{closed and bounded}). We also derive resolvent estimates in terms of the distance to the numerical range and discuss the role of dissipativity in this context.

Finally, we apply our results to operator pencils in Banach spaces. By associating a linear relation with an operator pencil, we obtain spectral enclosures in the spirit of \cite{GernMPT23,Wafa,GT21}. We further introduce two notions of numerical range for operator pencils in Banach spaces, extending the Hilbert space framework of \cite{BogliMarletta2019,Tretter2008} by means of the duality mapping. We establish a relationship between these numerical ranges (see Proposition~\ref{item:inclusion_wW}) and show, through a finite-dimensional example, that the numerical range of the associated linear relation may yield strictly sharper spectral enclosures than the numerical ranges of the operator pencil itself.

The paper is organized as follows. In Section~2 we recall basic notions and preliminary results on linear relations in Banach spaces. In Section~3 we introduce the numerical range for linear relations and study its fundamental properties together with spectral inclusions and resolvent estimates. Moreover, in Section~4, we apply our results to derive spectral enclosures for  operator pencils. Finally, in Section~5 we give an example of an operator where the numerical abscissa of the Banach space numerical range can be used to obtain a strict exponential decay of a~semigroup, which cannot be concluded by using the abscissa of the Hilbert space numerical range.

\section{Preliminaries on linear relations}
\label{sec:prelim}
\hskip0.5cm In this section, we gather some auxiliary notations and definitions that we will need in the
rest of the paper. Let $X$ and $Y$ be Banach spaces. Throughout, all Banach spaces are complex, and their dual spaces
consist of complex-linear functionals.  A linear relation $T:X\to Y$ is a mapping from a subspace $\dom T=\{x\in X:\ Tx\neq\emptyset\},$ called the domain of $T$, into the collection of nonempty subsets of $Y$ such that $$T(\alpha_1 x_1+\alpha_2 x_2)=\alpha_1T(x_1)+\alpha_2T(x_2),$$
for all nonzero scalars $\alpha_1,\alpha_2$ and all $x_1,x_2\in \dom T$.
We denote by $LR(X,Y)$ the class of all linear relations from $X$ to $Y$.
If $T$ maps each point in its domain to a singleton, then $T$ is said to be an operator. 
For the basic notions and properties of linear relations we refer to \textup{\cite{BergTW16,Cross98}}.
Given a linear relation $T\in LR(X,Y)$, we introduce the following sets
 \begin{align*}
    G(T)&=\{(x,y)\in X\times Y~:~ x\in\dom T,~ y\in Tx \},\\
    \ker T&=\{x\in X : (x,0)\in G(T)\},\\
    \ran T&=\{y\in Y : (x,y)\in G(T)\},\\
    \mul T&=\{y\in Y : (0,y)\in G(T)\}
\end{align*}
which are called the \textit{graph}, the \textit{kernel}, the \textit{range} and the \textit{multivalued} part
of $T$, respectively. We say that $T$ is \emph{injective} if $\ker(T)=\{0\}$, \emph{surjective} if $\ran(T)=Y$, and \emph{bijective} if $T$ is both injective and surjective. 
The identity relation defined on a subspace $M$ of $X$  is denoted by
$I_M$ (or simply $I$ when $M$ is understood). It is the linear relation whose graph is
$$
G(I_M)=\{(e,e)~:~ e\in M\}.
$$
The \textit{inverse} of the linear relation $T$ is given by
\begin{align*}
  G(T^{-1}):=\{(y,x)\in Y\times X : (x,y)\in G(T)\}.
\end{align*}
The linear relation $\alpha T$ with $\alpha\in\mathbb{C}$ is defined by
\begin{align}
\label{00}
G(\alpha T):=\{(x,\alpha y)\in X\times Y : (x,y)\in G(T)\}.
\end{align}
The sum of two linear relations $T,S\in LR(X,Y)$ is defined as
\begin{align}
\label{2}
G(T+S):=\{(x,y+y')\in X\times Y : (x,y)\in G(T), (x,y')\in G(S)\}.
\end{align}
If we assume that $X=Y$ then in view of (\ref{00}) and (\ref{2}) we have
\begin{align}
\label{2besser}
G(T-\lambda )=G(T-\lambda I)=\{(x,y-\lambda x):(x,y)\in G(T)\}.
\end{align}
If $M$ is a subspace of $X$ such that $M\cap \dom T\neq\emptyset$, then $T|_{M\cap \dom T}:=T|_{M}$ is defined by $$G(T|_{M}):=\{(x,y)\in G(T)~:~ x\in M\}.$$
The quotient map from $Y$ onto $Y/\cl(T(0))$ is denoted by $Q_T$. It is easy to see that $Q_TT$ is single-valued. This allows us to define the \textit{norm} of $Tx$ and of $T$, respectively, by
$$\|T x\|:=\|Q_TT x\|~~\mbox{for all}~x\in\dom  T~~\mbox{and}~\|T \|:=\|Q_TT \|.$$
Furthermore, according to \cite[Proposition II.1.4]{Cross98}, we have $$\|T x\|={\rm dist}(Tx,0)=\inf_{y\in Tx} \|y\|,~~~x\in\dom T.$$ 
A linear relation $T$ from $X$ into $Y$ is called \textit{closed} if its graph $G(T)$ is a closed subspace of $X\times Y$; the class of all closed linear relations from $X$ into $Y$ is denoted by $CR(X,Y)$. If $X=Y$, then we write $LR(X)$ and $CR(X)$ instead of $LR(X,X)$ and $CR(X,X)$, respectively. 
Let $X$ be a normed vector space. We denote by $X^{*}$ the norm dual of $X$ i.e., the space of all continuous functionals $x^*$ defined on $X$, with norm $$\|x^*\|=\sup_{\|x\|\leq 1} |x^*(x)|.$$
A consequence of Hahn--Banach theorem is the following.
\begin{thm}\label{hahnbanach}
Let $X$ be a normed vector space. Then, for all $x\in X$ and $x\neq 0$, there exists $x^*\in X^{*}$ such that $\|x^*\|=1$ and $x^*(x)=\|x\|$.
\end{thm}
If $M$ and $N$ are subspaces of $X$ and $X^*$, respectively, then 
$$M^{\bot}=\{x^*\in X^*~\mbox{such that}~x^*(x)=0, ~\mbox{for all}~x\in M\}$$
and $$N^{\top}=\{x\in X~\mbox{such that}~x^*(x)=0, ~\mbox{for all}~x^*\in N\}$$
\begin{defn}\cite[Definition III.1.1]{Cross98}
The \emph{adjoint} $T^\prime$ of $T$ is defined by
$$G(T^\prime):= G(-T^{-1})^{\bot}\subset Y^*\times X^*, $$
So that $(y^\prime,x^\prime)\in G(T^\prime)$ if and only if $y^\prime(y)=x^\prime(x)$ for all $(x,y)\in G(T).$
\end{defn}
To discuss the results of the classical spectral theory associated with the numerical range, we first introduce the notion of resolvent and spectrum for linear relations.
\begin{defn}\cite {GT21}
Let $X$ be a Banach space and $T\in CR(X)$. A complex number $\lambda\in\mathbb{C}$ is called a \emph{resolvent point} of $T$ if
$$
\ker(T-\lambda)=\{0\} 
\quad \text{and} \quad 
\operatorname{ran}(T-\lambda)=X.
$$
In this case, $(T-\lambda)^{-1}$ is the graph of a bounded linear operator and it is called the \emph{resolvent} of $T$ at $\lambda$. The set of all resolvent points is denoted by  $\rho(T)$. 
Consequently, it is a bounded linear operator on $X$ by the closed graph theorem. 
The \emph{spectrum} of $T$ is defined by
$\sigma(T):=\mathbb{C}\setminus\rho(T).$ 
The \emph{point spectrum} $\sigma_{\rm p}(T)$ of $T$ is defined by
$$\sigma_{\rm p}(T)
=\left\{\lambda\in\mathbb{C}:\ker(T-\lambda)\neq\{0\}\right\}.$$
The \emph{approximate point spectrum}  $\sigma_{\rm ap}(T)$ of $T$ is defined by$$
\sigma_{\rm ap}(T)=\left\{\lambda\in\mathbb{C}~:~\exists (x_n)\subset \dom(T),\ \|x_n\|=1,\ \|(T-\lambda)x_n\|\to 0 \right\}.$$
Clearly, the following holds $\sigma_{\rm p}(T)\subseteq \sigma_{\rm ap}(T)\subseteq\sigma(T)$.
\end{defn}

\section{The Banach space numerical range for a linear relation}
We first present the basic definitions of the Banach space numerical range.
\begin{defn}
Let $X$ be a Banach space and let $T\subset X\times X$ be a linear relation.
For $x\in X$, let
\[
J_X(x):=\{x^*\in X^*:\|x^*\|=\|x\|,\;x^*(x)=\|x\|^2\}
\]
denote the normalized duality mapping.
The numerical range of $T$ is defined by
\[
W(T):=\left\{\frac{x^*(y)}{\|x\|^2}:\;(x,y)\in G(T),\;x\neq0,\;x^*\in J_X(x)\right\}.
\]
\end{defn}
If $X$ is a  Hilbert space, then using the Riesz representation theorem we find that the numerical range of $T\in LR(X)$ coincides with the definition given by \cite{Rofe90}, namely
$$ W(T)=\{\langle y,x\rangle : (x,y)\in G(T),\ \|x\|=1\}.$$

\subsection{Elementary properties}
The proposition below provides conditions for which the numerical range is non-empty. This extends the corresponding result for operators from \cite{BD1} and \cite{BD} to linear relations. 
\begin{prop}
Let $T,S\in LR(X)$ and $\alpha,\beta\in \mathbb{C}$. Then
\begin{enumerate}
  \item $W(\alpha I+T)=\alpha+ W(T)$.
 \item $W(\alpha T+\beta S)\subseteq \alpha W(T)+\beta W(S)$.
  \item If $X$ is a non-trivial Banach space with $\dom T\neq\{0\}$, then $W(T)$ is non-empty.
\end{enumerate}
\end{prop}
\begin{proof} 
  \begin{enumerate}
  \item Let $T\in LR(X)$, then for $(x,y)\in G(\alpha I+T)$  there exists $u\in X$ such that $y=\alpha x+u$ and 
using the linearity of $x^*$, we have
\begin{eqnarray*}
    W(\alpha I+T)&=&\left\{\alpha\frac{x^*(x)}{\|x\|^2}+\frac{x^*(u)}{\|x\|^2}: (x,u)\in G(T),\;x\neq0,\;x^*\in J_X(x)\right\}\\
    &=&\alpha+\left\{\frac{x^*(u)}{\|x\|^2}: (x,u)\in G(T),\;x\neq0,\;x^*\in J_X(x)\right\}\\
     &=&\alpha+W(T).
\end{eqnarray*}
 \item 
 Let $\lambda\in W(\alpha T+\beta S)$. Then, $\exists x\in \dom T$ and $(x,y)\in G(\alpha T+\beta S)$ satisfying 
  \begin{equation}\label{egalite}
     \lambda=\frac{x^*(y)}{\|x\|^2}~\text{such that}~x\neq0,\;x^*\in J_X(x). 
  \end{equation}
  The fact that $(x,y)\in G(\alpha T+\beta S)$ implies for \eqref{2}, that there exist $u,v\in X$ such that $(x,\alpha u)\in G(\alpha T)$ and $(x,\beta v)\in G(\beta S)$ with $\alpha u+\beta v=y$. 
  The use of \eqref{egalite} allows us to deduce that 
  \begin{eqnarray*}
      \lambda
=\frac{x^*(\alpha u)}{\|x\|^2}+\frac{x^*(\beta v)}{\|x\|^2}&=\alpha\frac{x^*(u)}{\|x\|^2}+\beta\frac{x^*(v)}{\|x\|^2}\in \alpha W(T)+\beta W(S).
  \end{eqnarray*}
 \item Since $\dom T \neq \{0\}$, there exist nonzero 
$x \in \dom T$ with $(x,y)\in G(T)$. By the Hahn--Banach theorem, there exists $x^*\in X^*$ satisfying
$\|x^*\|=\|x\|~\text{and}~x^*(x)=\|x\|^2,$ that is, $x^*\in J_X(x)$. Therefore, by the definition of $W(T)$, we have $\tfrac{x^*(y)}{\|x\|^2}\in W(T)$. Consequently, $W(T)\neq\emptyset$.
\end{enumerate}  
\end{proof}

The following theorem provides some conditions, when $W(T)=\mathbb{C}$ is possible and it is an extension of \cite{Rofe90} to the Banach space setting.
\begin{thm}
\label{closed and bounded}
Let $T\in LR(X)$. Then the following assertions hold.
\begin{enumerate}
    \item If there exist $y_0\in\mul T$, $x_0\in\dom T\setminus\{0\}$, and $x_0^*\in J_X(x_0)$
    such that $x_0^*(y_0)\neq0$, then
 $W(T)=\mathbb{C}$.

    \item If $\dim(\dom T)<\infty$,
    then either
$W(T)=\mathbb{C}$ 
    or $W(T)$ is compact.
\end{enumerate}
\end{thm}
\begin{proof}
\begin{enumerate}
\item
Let $y_0\in\mul T$,
$x_0\in\dom T\setminus\{0\}$, $x_0^*\in J_X(x_0)$ 
and suppose that $x_0^*(y_0)\neq0$.
Choose an element $y_1\in Tx_0$. Since $y_0\in\mul T$, we have $(0,y_0)\in G(T)$. Since $G(T)$ is a linear subspace, it follows that
\[
    (x_0,y_1+\mu y_0)\in G(T)
    \qquad
    \text{for every }\mu\in\mathbb{C}.
\]
Consequently,
\[
    \frac{x_0^*(y_1+\mu y_0)}{\|x_0\|^2}
    =
    \frac{x_0^*(y_1)}{\|x_0\|^2}
    +
    \mu\frac{x_0^*(y_0)}{\|x_0\|^2}
    \in W(T).
\]
Since $\tfrac{x_0^*(y_0)}{\|x_0\|^2}\neq0$,
the affine mapping
\[
    \mu
    \mapsto
    \frac{x_0^*(y_1)}{\|x_0\|^2}
    +
    \mu\frac{x_0^*(y_0)}{\|x_0\|^2}
\]
maps $\mathbb{C}$ onto $\mathbb{C}$. Hence, 
    $W(T)=\mathbb{C}$.
\item
Set $D:=\dom T$ and $M:=\mul T$. If $D=\{0\}$, then $W(T)=\emptyset$, which is compact. We may therefore assume that $D\neq\{0\}$. Since $D$ is finite-dimensional, we can choose a linear map $A_0:D\longrightarrow X$ such that
\[
    A_0x\in Tx
    \qquad
    \text{for every }x\in D.
\]
For example, choose a basis $d_1,\ldots,d_n$ of $D$, select
$y_j\in Td_j$, and define
\[
    A_0\left(\sum_{j=1}^n\alpha_jd_j\right)
    :=
    \sum_{j=1}^n\alpha_jy_j.
\]
Because $D$ is finite-dimensional, the linear map $A_0$ is
bounded.

For every $x\in D$, we have $Tx=A_0x+M$. Indeed, if $y\in Tx$, then
\[
    (x,y)-(x,A_0x)=(0,y-A_0x)\in G(T),
\]
and hence $y-A_0x\in M$. Conversely, if $m\in M$, then
\[
    (x,A_0x)+(0,m)=(x,A_0x+m)\in G(T).
\]

We distinguish two cases.

\medskip
\noindent
\textbf{Case 1.}
Suppose that there exist $x\in D\setminus\{0\}$,
    $x^*\in J_X(x)$, $m\in M$ such that $x^*(m)\neq0$. Then 1. implies that $W(T)=\mathbb{C}$.\\
\noindent
\textbf{Case 2.}
Suppose that $x^*(m)=0$ for every 
$x\in D\setminus\{0\}$,
$x^*\in J_X(x)$, and 
$m\in M$. If $y\in Tx$, then $y=A_0x+m$ for some $m\in M$, and therefore
\[
    x^*(y)=x^*(A_0x).
\]
By homogeneity, it is sufficient to consider vectors of norm one.
Let $S_D:=\{x\in D:\|x\|=1\}$. Then
\[
    W(T)
    =
    \left\{
        x^*(A_0x):
        x\in S_D,\;
        x^*\in J_X(x)
    \right\}.
\]
Equip $S_D$ with the norm topology and the closed dual unit ball
\[
    B_{X^*}:=\{x^*\in X^*:\|x^*\|\leq1\}
\]
with the weak-* topology. Since $D$ is finite-dimensional,
$S_D$ is compact. By the Banach--Alaoglu theorem,
$B_{X^*}$ is weak-* compact. Consider the set
\[
    K
    :=
    \left\{
        (x,x^*)\in S_D\times B_{X^*}:
        x^*(x)=1
    \right\}.
\]
For $x\in S_D$, the conditions $x^*\in B_{X^*}$,
$x^*(x)=1$ imply $\|x^*\|=1$, 
and hence $x^*\in J_X(x)$.
Therefore, $K
    =
    \left\{
        (x,x^*):
        x\in S_D,\;
        x^*\in J_X(x)
    \right\}$. 
The mapping $(x,x^*)\mapsto x^*(x)$ is continuous on $S_D\times B_{X^*}$ with respect to the norm
topology in the first component and the weak-* topology in the
second component. Hence, $K$ is a closed subset of the compact
space $S_D\times B_{X^*}$, and is therefore compact. Finally, consider
\[
    \Phi:K\rightarrow\mathbb{C},
    \qquad
    \Phi(x,x^*)=x^*(A_0x).
\]
Since $A_0$ is bounded, the mapping $\Phi$ is continuous for the
same product topology. Indeed, if $x_\alpha\rightarrow x$
in norm and $x_\alpha^*\rightarrow x^*$ in the weak-* topology, then
\begin{align*}
    \left|
        x_\alpha^*(A_0x_\alpha)-x^*(A_0x)
    \right|
    &\leq
    \left|
        x_\alpha^*(A_0(x_\alpha-x))
    \right|
    +
    \left|
        (x_\alpha^*-x^*)(A_0x)
    \right|
    \\
    &\leq
    \|A_0(x_\alpha-x)\|
    +
    \left|
        (x_\alpha^*-x^*)(A_0x)
    \right|
    \longrightarrow0.
\end{align*}
Consequently, $ W(T)=\Phi(K)$
is the continuous image of a compact set and is therefore compact. 
Combining the two cases, we conclude that either
$W(T)=\mathbb{C}$
or $W(T)$ is compact.
\end{enumerate}
\end{proof}

\subsection{Spectral inclusion}
The following theorem is one of the main results of this paper, where we show that the numerical range has the spectral inclusion property. 
\begin{thm}\label{theorem1}
Let $T\in CR(X)$. Then
\begin{enumerate}
  \item $\sigma_{\rm p}(T)\subseteq W(T)$;
  \item $\sigma_{\rm ap}(T)\subseteq \cl(W(T))$;
  \item $\sigma(T)\subseteq \cl(W(T))\cup W(T')$. 
\end{enumerate}
\end{thm}
\begin{proof}
\begin{enumerate}
   \item Let $\lambda\in\sigma_{\rm p}(T)$. Then, there exist $x\neq 0$ and
$x^*\in J_X(x)$ such that $(x,0)\in G(T-\lambda)$.\\ 
   Since, $(x,0)\in G(T-\lambda)$ implies that there exists $y\in X$ such that
   $(x,0)=(x,y-\lambda x)$ satisfying $(x,y)\in G(T)$.\\
    therefore 
\[
\frac{x^*(y)}{\|x\|^2}
=\frac{x^*(\lambda x)}{\|x\|^2}
=\lambda\frac{x^*(x)}{\|x\|^2}
=\lambda.
\]
This implies that $\lambda\in W(T).$\\ 
   \item Let $\lambda\in\sigma_{\rm ap}(T)$. Then there exists
$(x_n)\subseteq\dom T$ such that
\[
    \|x_n\|=1,
    \qquad
    \|(T-\lambda I)x_n\|\rightarrow0.
\]
For every $n$, choose $y_n\in Tx_n$ such that
\[
    \|y_n-\lambda x_n\|
    \le
    \|(T-\lambda I)x_n\|+\frac1n.
\]
Choose $x_n^*\in J_X(x_n)$. Then
\[
\left|
    x_n^*(y_n)-\lambda
\right|
\le
\|y_n-\lambda x_n\|
\rightarrow0.
\]   
Consequently, $\lambda\in \cl(W(T))$.
   \item 
   Let $\lambda\in\sigma(T)\setminus\overline{W(T)}$. By 2.,
$\lambda\notin\sigma_{\rm ap}(T)$. Hence, $T-\lambda I$ is
bounded below, i.e.\ there exists $c>0$ such that 
$\|(T-\lambda)x\|\geq c\|x\|$ for all $x\in\dom T$. Since $T-\lambda I$ is closed, by
\cite[Theorem~III.4.2(b)]{Cross98}, $\ran(T-\lambda)$ is closed. 
Moreover, $T-\lambda I$ is injective. Since
$\lambda\in\sigma(T)$, the range of $T-\lambda I$ is a proper
closed subspace of $X$. By the Hahn--Banach theorem, there exists $0\neq f\in X^*$ such
that
\[
    f(y-\lambda x)=0
    \qquad
    \text{for all }(x,y)\in G(T).
\]
Thus,
\[
    f(y)=\lambda f(x)
    \qquad
    \text{for all }(x,y)\in G(T),
\]
and hence $(f,\lambda f)\in G(T')$. Therefore, $\lambda\in\sigma_{\rm p}(T')\subseteq W(T')$.
\end{enumerate}
\end{proof}
In the following, we present an example that the numerical range of adjoint in Theorem~\ref{theorem1} is necessary.
\begin{ex}
Let $X=\mathbb{C}^2$
with the Euclidean norm, and let
$e_1=\left(
    \begin{smallmatrix}
        1\\0
    \end{smallmatrix}\right)$, $e_2=\left(
    \begin{smallmatrix}
        0\\1
    \end{smallmatrix}\right)$.
Consider the linear relation $T$ defined by
\[
    G(T)
    =
    \left\{
        \left(\alpha e_1,0\right):
        \alpha\in\mathbb{C}
    \right\}.
\]
Thus, $T$ is the zero operator with proper domain $\dom T=\operatorname{span}\{e_1\}$. 
In particular, $T$ is closed and $\mul T=\{0\}$. 
For every $(x,y)\in G(T)$, we have $y=0$. Hence, for every
$x\in\dom T\setminus\{0\}$ and every $x^*\in J_X(x)$, $\tfrac{x^*(y)}{\|x\|^2}=0$. Consequently, $W(T)=\{0\}$. We next determine the spectrum of $T$. Since
\[
    \ker T=\operatorname{span}\{e_1\}\neq\{0\},
\]
we have $0\in\sigma_{\rm p}(T)$. For $\lambda\neq0$,
\[
    G(T-\lambda I)
    =
    \left\{
        \left(\alpha e_1,-\lambda\alpha e_1\right):
        \alpha\in\mathbb{C}
    \right\}.
\]
Thus, $T-\lambda I$ is injective, but
$\ran(T-\lambda I)
    =
    \operatorname{span}\{e_1\}
    \neq X$. Therefore, $\lambda\notin\rho(T)$ for every $\lambda\neq0$, and
hence $\sigma(T)=\mathbb{C}$.

It remains to compute the Banach adjoint relation. Let
$e_1^*,e_2^*\in X^*$ denote the coordinate functionals,
\[
    e_1^*(z_1,z_2)=z_1,
    \qquad
    e_2^*(z_1,z_2)=z_2.
\]
By the definition of the Banach adjoint, a pair
$(f,g)\in X^*\times X^*$ belongs to $G(T')$ if and only if $ f(y)=g(x)$
for all  $(x,y)\in G(T)$. Since every element of $G(T)$ has the form $(\alpha e_1,0)$,
this condition is equivalent to
\[
    0=g(\alpha e_1)
    \qquad
    \text{for every }\alpha\in\mathbb{C},
\]
or, equivalently, $g(e_1)=0$. Therefore,
\[
    G(T')
    =
    \left\{
        (f,g)\in X^*\times X^*:
        g\in\operatorname{span}\{e_2^*\}
    \right\}
    =
    X^*\times\operatorname{span}\{e_2^*\}.
\]
In particular,
\[
    \dom T'=X^*,
    \qquad
    \mul T'=\operatorname{span}\{e_2^*\}.
\]

We now show that $W(T')=\mathbb{C}$. Let $\mu\in\mathbb{C}$ be arbitrary and set $f=e_2^*$,$g=\mu e_2^*$. Then $(f,g)\in G(T')$. Let
\[
    \widehat{e}_2\in X^{**},
    \qquad
    \widehat{e}_2(h)=h(e_2),
    \quad h\in X^*,
\]
be the canonical image of $e_2$ in $X^{**}$. Since
\[
    \|f\|=\|\widehat{e}_2\|=1,
    \qquad
    \widehat{e}_2(f)=e_2^*(e_2)=1,
\]
we have $\widehat{e}_2\in J_{X^*}(f)$. Consequently,
\[
    \frac{\widehat{e}_2(g)}{\|f\|^2}
    =
    \widehat{e}_2(\mu e_2^*)
    =
    \mu.
\]
Since $\mu\in\mathbb{C}$ was arbitrary, it follows that $W(T')=\mathbb{C}$. We therefore obtain
\[
    \overline{W(T)}
    =
    \{0\}
    \neq
    \sigma(T)
    =
    \mathbb{C},
\]
whereas
\[
    \overline{W(T)}\cup W(T')
    =
    \mathbb{C}
    =
    \sigma(T).
\]
\end{ex}
\begin{rem}
Since $X^*$ denotes the space of complex-linear continuous
functionals, for the Banach adjoint relation one has $(T-\lambda I)'=T'-\lambda I$. 
Thus, no complex conjugation occurs in the spectral inclusion.

If $X$ is a Hilbert space with an inner product that is linear in the first entry and $T^*$ denotes the Hilbert-space
adjoint, then
\[
    (T-\lambda I)^*=T^*-\overline{\lambda}I.
\]
Under the conjugate-linear Riesz identification of $X$ with $X^*$,
\[
    W_{X^*}(T')
    =
    \{\overline{z}:z\in W_X(T^*)\}.
\]
Hence, in terms of the Hilbert-space adjoint, the corresponding
spectral enclosure is
\[
    \sigma(T)
    \subseteq
    \overline{W_X(T)}
    \cup
    \{\overline{z}:z\in W_X(T^*)\}.
\]
\end{rem}
\begin{rem}
\label{rem:bounded}
For bounded operators on Banach spaces, there is a close relationship between the numerical range of an operator $T$ and that of its Banach adjoint $T'$
\[
    W(T)
    \subseteq
    W(T')
    \subseteq
    \cl(W(T)),
\]
see, for example, \cite[Corollary 3]{BD73}.  
In particular, $\cl(W(T))=\cl(W(T'))$, see \cite[Theorem 2]{Bollobas1970}. Moreover, if $X$ is reflexive, then 
$W(T)=W(T')$, see  \cite[Proposition~2.6]{Jahedi2012}. These assertions do not extend to linear relations in general. Indeed, for the linear relation $T=X\times X$, one has
$W(T)=\mathbb C$, whereas
$T'=\{(0,0)\}$ and hence $W(T')=\emptyset$.
\end{rem}

\begin{cor}
Let $T\in CR(X)$ and $\dom T=X$. Then
\[
    \sigma(T)\subseteq\cl(W(T)).
\]
\end{cor}
\begin{proof}
    If $\mul T\neq\{0\}$, then we obtain from Theorem~\ref{closed and bounded} that $W(T)=\mathbb{C}$ and therefore the stated inclusion holds trivially. If  $\mul T=\{0\}$ then $T$ is the graph of a bounded operator and hence, from Remark~\ref{rem:bounded} and Theorem~\ref{theorem1} we obtain the stated inclusion. 
\end{proof}

\subsection{Resolvent estimates}
The resolvent family of a~linear relation is a~pseudo-resolvent and estimates on these resolvents are useful to solve differential-algebraic equations~\cite{Gernandt2025,JacobMorris2022}. We now establish a resolvent estimate in terms of the numerical range in the Banach space setting, extending the classical Hilbert space result of Kato~\cite[Theorem V.3.2]{Kato95}.
\begin{thm}
Let $X$ be a Banach space and let $T\in CR(X)$. Assume that $\lambda \in\rho(T)\setminus\cl(W(T))$. Then
\begin{align}
\label{eq:resolvent_estim}
\|(T-\lambda)^{-1}\|\le \frac{1}{\dist(\lambda,W(T))}.
\end{align}
\end{thm}
\begin{proof}
Since $\lambda \in \rho(T)$ then $T-\lambda $ is bijective and its inverse $(T-\lambda)^{-1}$ is bounded. 
Let $(x,y)\in G(T)$ with $x\neq0$ and  $x^*\in J_X(x)$. By the definition of $W(T)$, we have
$\tfrac{x^*(y)}{\|x\|^2}\in W(T)$ and using the definition of the distance from $\lambda$ to $W(T)$ it follows that
\begin{align}
\label{eq:distance_estimate}   
\frac{\|y-\lambda x\|}{\|x\|}= \frac{\|x^*\|\,\|y-\lambda x\|}{\|x\|^2} \geq \left|\frac{x^*(y-\lambda x)}{\|x\|^2}\right|=\left|
\frac{x^*(y)}{\|x\|^2}-\lambda
\right|\geq \dist(\lambda, W(T)).
\end{align}
 Hence, by taking the infimum over all such $y$ in \eqref{eq:distance_estimate},  we obtain 
 \begin{equation}\label{estim}
  \|(T-\lambda)x\|=\inf_{y\in Tx}\|y-\lambda x\|\geq \dist(\lambda, W(T))\,\|x\|\qquad x\in\dom(T)   
 \end{equation}
Now let $z\in X$ and set $x=(T-\lambda)^{-1}z$.
Since $\lambda\in\rho(T)$, the relation $T-\lambda$ is injective. Hence, 
$(T-\lambda)^{-1}$ is single-valued. Since $x=(T-\lambda)^{-1}z,$ we have
$(z,x)\in G((T-\lambda)^{-1})$.
By definition of the inverse relation,
$(x,z)\in G(T-\lambda)$.  So, from \eqref{2besser}, there exists $u\in T(x)$ such that $z=u-\lambda x$. Therefore, $\|z\|=\|u-\lambda x\|$.  Applying  \eqref{estim}, we obtain  \[
\|z\|\geq \dist(\lambda,W(T))\,\|x\|.
\]
Using $x=(T-\lambda)^{-1}z$, we obtain
\[
\|(T-\lambda)^{-1}z\|
=\|x\|
\leq \frac{1}{\dist(\lambda,W(T))}\,\|z\|.
\]
For all $z$,  we conclude \eqref{eq:resolvent_estim}.
 \end{proof}
 \begin{rem}
A linear relation $T\subseteq X\times X$ is called
\emph{dissipative} if, for every $(x,y)\in G(T)$, there exists
$x^*\in J_X(x)$ such that
\[
    \operatorname{Re}x^*(y)\leq0.
\]
It is called \emph{m-dissipative} if it is dissipative and
$\operatorname{ran}(\lambda I-T)=X$ for some $\lambda>0$.
In this case,
\[
    (0,\infty)\subseteq\rho(T),
    \qquad
    \|(\lambda I-T)^{-1}\|\leq\frac1\lambda,
    \qquad \lambda>0,
\]
\cite[Proposition~4.3]{ArendtChalendarMoletsane2024}.
In general, this estimate cannot be recovered directly from
$\operatorname{dist}(\lambda,W(T))$, because dissipativity requires
the existence of one norming functional for each graph element,
whereas $W(T)$ contains the values associated with all norming
functionals.
\end{rem}

\section{Application to spectral inclusions for operator pencils}
In this section, we apply the spectral inclusion for linear relations obtained in Theorem~\ref{theorem1} to obtain spectral inclusions for operator pencils $\mathcal{A}(\lambda)=\lambda E-A$, where $E:X\rightarrow X$ is a bounded operator in the Banach space $X$ and $A:X\supseteq\dom A\rightarrow X$ is closed and densely defined. Recall that the spectrum and the point spectrum of an operator pencil $\mathcal{A}(\lambda)=\lambda E-A$ are given by 
\begin{align*}
\sigma(\mathcal{A})&=\{\lambda\in\mathbb{C} ~:~ \lambda E-A \text{ not bijective}\},\\
\sigma_{\rm p}(\mathcal{A})&=\{\lambda\in\mathbb{C} ~:~ \ker (\lambda E-A)\neq \{0\}\}.
\end{align*}
Furthermore, we assume throughout this section that the resolvent set $\rho(\mathcal{A}):=\mathbb{C}\setminus\sigma(\mathcal{A})$ is not empty.

In this framework, one associates with a pencil 
$\mathcal{A}(\lambda) = \lambda E - A$ the range representation
\[
T = \mathrm{ran}
\begin{bmatrix}
E \\
A
\end{bmatrix}=\{(Ex,Ax)~:~ x\in\dom A\},
\]
which enables a reduction of questions on linear relations to the analysis of operator pencils. It was shown in \cite[Proposition 3.4]{GT21} that 
\[
\sigma_{\rm p}(\mathcal{A})=\sigma_{\rm p}(T),\quad \sigma(\mathcal{A})=\sigma(T).
\]
Furthermore, since $\rho(\mathcal{A})\neq\emptyset$ we have that $\rho(T)\neq\emptyset$ and therefore $T$ defines a closed linear relation.

Hence, as a combination of  \cite[Proposition 3.4]{GT21} and Theorem~\ref{theorem1}, we obtain the following spectral inclusion
\begin{prop}
\label{prop:enclosure_relation}
Let $E:X\rightarrow X$ be bounded in the Banach space $X$ and $A:X\supseteq\dom A\rightarrow X$ is closed and densely defined and consider the operator pencil $\mathcal{A}(\lambda)=\lambda E-A$ and the linear relation 
$T = \mathrm{ran}\left[\begin{smallmatrix} E \\ A \end{smallmatrix}\right]$ with $\rho(\mathcal{A})\neq\emptyset$.
Then $T\in CR(X)$ with $\rho(T)\neq\emptyset$ and 
\begin{align}
    \label{new_spectral_inclusion}
\sigma_{\rm p}(\mathcal{A})\subseteq W(T),\quad \sigma(\mathcal{A})\subseteq \cl(W(T))\cup W(T').
\end{align}
\end{prop}
We give an example that this can lead to better estimates than using the numerical range of operator pencils. 
The main idea is that the behavior of $T$ can be analyzed through the family of numerical ranges $W(\mathcal{A}(\lambda))$.

In the Hilbert space setting, numerical ranges of operator pencils $A - \lambda E$ were introduced in  \cite{BogliMarletta2019} using inner product structures. On the other hand, in Banach spaces, the numerical range of a bounded operator was defined in \cite{Lumer61} through duality. Combining these approaches, we propose an extension of the numerical range to operator pencils in Banach spaces.

\begin{defn}
Let $E:X\rightarrow X$ be bounded between Banach space $X$  and $A:X\supseteq\dom A\rightarrow X$ is closed. Then we define the numerical ranges associated with the operator pencil $\lambda E-A$ 
\begin{align*}
W(A,E) &:= \left\{ \lambda \in \mathbb{C} ~\colon~ 0 \in \cl(W( \lambda E-A)) \right\},\\
w(A,E)&:= \left\{\frac{x^*(Ax)}{x^*(Ex)}
:\;x\in\dom(A),\;x^*\in J_X(x),\;x^*(Ex)\neq0
 \right\}.
\end{align*}
\end{defn}
If $X$ is a Hilbert space with inner product $\langle\cdot,\cdot\rangle$, we recover the classical definition of numerical range of operator pencils, see e.g. \cite{BogliMarletta2019,Tretter2008},
\[
w(A,E)=\left\{\frac{\langle Ax,x\rangle}{\langle Ex,x\rangle} : \langle Ex,x\rangle\neq 0, x\in\dom A\right\}.
\]
In the following proposition, we study the relationship between the Banach space numerical ranges of operator pencils. 
\begin{prop} \label{item:inclusion_wW}
Let $E:X\rightarrow X$ be bounded in the Banach space $X$  and $A:X\supseteq\dom A\rightarrow X$ is closed and densely defined. 
We have $$w(A,E)\subseteq W(A,E)$$
\end{prop}
\begin{proof}
Let $\lambda\in w(A,E)$, then there exist $x \in\dom(A)\setminus\{0\}$ and $x^*\in J_X(x)$ such that $x^*(Ex)\neq0$ and $\lambda=\frac{x^*(Ax)}{x^*(Ex)}.$\\
Hence,
\[
x^*((\lambda E-A)x)=\lambda x^*(Ex)-x^*(Ax) = 0.
\]
Therefore, $0=\frac{x^*((\lambda E-A)x)}{\|x\|^2}
\in W(\lambda E-A).$ It follows that $0 \in\cl(W(\lambda E-A))$. Consequently, $\lambda \in W(A,E)$. 
 \end{proof}

If $X$ is a Hilbert space, then it was shown under additional assumptions in \cite{BogliMarletta2019} that $\sigma(\mathcal{A})\subseteq W(A,E)$. In Banach spaces, we have 
$\lambda\in\sigma(\mathcal{A})$ if and only if $0\in\sigma(\lambda E-A)\subseteq \cl(W(\lambda E-A))\cup W(\lambda E'-A')$, thus proving that 
\[
\sigma(\mathcal{A})\subseteq W(A,E)\cup W(A',E').
\]
In the following, we give a finite-dimensional example, which shows that the spectral enclosures for operator pencils by the numerical range of the linear relation based on Proposition~\ref{prop:enclosure_relation} are better.

\begin{ex}
Let $X=(\mathbb{C}^2,\|\cdot\|_p)$, $1<p<\infty$, 
and let $q\in(1,\infty)$ be the conjugate exponent, that is, 
$\tfrac1p+\tfrac1q=1$. For some $a\in\mathbb{R}$, consider the matrix pencil 
\[
\mathcal{A}(\lambda)=\lambda E-A,\quad 
    E=
    \begin{pmatrix}
        0&0\\
        0&1
    \end{pmatrix},
    \qquad
    A=
    \begin{pmatrix}
        1&0\\
        0&a
    \end{pmatrix}.
\]
We first compute the spectrum of the pencil. Since
\[
    \lambda E-A
    =
    \begin{pmatrix}
        -1&0\\
        0&\lambda-a
    \end{pmatrix},
\]
the operator $\lambda E-A$ is invertible if and only if
$\lambda\neq a$. Consequently,
\[
    \sigma(\mathcal{A})=\{a\}.
\]
The linear relation associated with the pencil is
\[
    T
    =
    \operatorname{ran}
    \begin{bmatrix}
        E\\ A
    \end{bmatrix}
    =
    \left\{
    \left(
        \begin{pmatrix}
            0\\x_2
        \end{pmatrix},
        \begin{pmatrix}
            x_1\\ax_2
        \end{pmatrix}
    \right)
    :
    x_1,x_2\in\mathbb{C}
    \right\}.
\]
In particular, $\dom T=\operatorname{span}\{e_2\}$,
$\mul T=\operatorname{span}\{e_1\}$. We next compute the numerical range of $T$. Let $x=
\left(    \begin{smallmatrix}
        0\\x_2
    \end{smallmatrix}\right)
    \in\dom T\setminus\{0\}$.
Since $1<p<\infty$, the normalized duality mapping on $\mathbb{C}^2$ with $\|\cdot\|_p$
is single-valued. Every $x^*\in J_X(x)$ is of the form
\[
    x^*=
    \begin{pmatrix}
        0\\\overline{x_2}
    \end{pmatrix},
\]
where we identify $X^*$ with $\mathbb{C}^2$ equipped with $\|\cdot\|_q$. For every $y=
    \left(\begin{smallmatrix}
        x_1\\ax_2
    \end{smallmatrix}\right)
    \in Tx$,
we therefore obtain
\[
    \frac{x^*(y)}{\|x\|_p^2}
    =
    \frac{a|x_2|^2}{|x_2|^2}
    =a.
\]
Hence,
\[
    W(T)=\{a\}.
\]
Let us also determine the Banach adjoint relation $T'$. A pair $(f,g)\in X^*\times X^*$ belongs to $G(T')$ if and only if
\[
    f(y)=g(x)
    \qquad
    \text{for all }(x,y)\in G(T).
\]
Writing $f=\left(\begin{smallmatrix}
f_1\\ f_2    
\end{smallmatrix}\right)$ and 
  $g=\left(\begin{smallmatrix}
      g_1\\ g_2
  \end{smallmatrix}\right)$,
this condition becomes
\[
    f_1x_1+af_2x_2=g_2x_2
    \qquad
    \text{for all }x_1,x_2\in\mathbb{C}.
\]
Thus, $f_1=0$, $g_2=af_2$, whereas $g_1$ is arbitrary. Therefore,
\[
    T'
    =
    \left\{
    \left(
        \begin{pmatrix}
            0\\f_2
        \end{pmatrix},
        \begin{pmatrix}
            g_1\\af_2
        \end{pmatrix}
    \right)
    :
    f_2,g_1\in\mathbb{C}
    \right\}.
\]
The same argument as above, now in $X^*$, yields
\[
    W(T')=\{a\}.
\]
Consequently, the spectral enclosure obtained from the associated
linear relation is exact:
\[
    \sigma(\mathcal{A})
    =
    \overline{W(T)}\cup W(T')
    =
    \{a\}.
\]
For $x=
    \begin{pmatrix}
        x_1&x_2
    \end{pmatrix}^\top$ with $x_2\neq0,$
the normalized duality functional is
\[
    x^*
    =
    \|x\|_p^{2-p}
    \begin{pmatrix}
        |x_1|^{p-2}\overline{x_1}\\
        |x_2|^{p-2}\overline{x_2}
    \end{pmatrix}.
\]
Hence,
\[
    x^*(Ax)
    =
    \|x\|_p^{2-p}
    \left(
        |x_1|^p+a|x_2|^p
    \right),\qquad
    x^*(Ex)
    =
    \|x\|_p^{2-p}|x_2|^p.
\]
It follows that
\[
    \frac{x^*(Ax)}{x^*(Ex)}
    =
    a+\frac{|x_1|^p}{|x_2|^p}.
\]
Since the quotient $|x_1|^p/|x_2|^p$ may take any value in
$[0,\infty)$, we obtain
\[
    w(A,E)=a+[0,\infty).
\]

The same set is obtained from
$W(A,E)$.  Indeed, for $x\neq0$ and $x^*\in J_X(x)$,
\[
    \frac{x^*((A-\lambda E)x)}{\|x\|_p^2}
    =
    \frac{
        |x_1|^p+(a-\lambda)|x_2|^p
    }{
        |x_1|^p+|x_2|^p
    }.
\]
Thus, $0\in W(\lambda E-A)$ if and only if $|x_1|^p+(a-\lambda)|x_2|^p=0$ for some $x_2\neq0$, which is equivalent to
\[
    \lambda
    =
    a+\frac{|x_1|^p}{|x_2|^p}.
\]
Note that since the space $X$ is finite-dimensional the numerical ranges $W(\lambda E-A)$ are compact and therefore the additional closure in the definition of $W(A,E)$ can be neglected.  
Therefore,
\[
    W(A,E)=a+[0,\infty)=w(A,E),
\]
whereas  $\sigma(\mathcal{A})
    =
    W(T)
    =
    W(T')
    =
    \{a\}$.
Hence, the numerical ranges of the associated linear relation provide
an exact spectral enclosure, while the classical numerical ranges of
the pencil produce an unbounded enclosure.
\end{ex}
\section{An illustration using numerical abscissas and semigroup estimates}
We present another example where the use of the Banach space numerical range leads to exponential estimates for semigroup generators via the numerical abscissa.  

Let $\Omega=(0,1)$ and, for $p\in\{1,2\}$, define
\[
    X_p=L^p(\Omega;(\mathbb{C}^2,\|\cdot\|_p)),\quad     \|f\|_{X_p}
    =
    \left(
        \int_0^1
        \left(
            |f_1(\xi)|^p+|f_2(\xi)|^p
        \right)
        \,\mathrm{d}\xi
    \right)^{1/p}.
\]
Consider the matrix
\[
    M=
    \begin{pmatrix}
        -1&9\\
        0&-10
    \end{pmatrix}
\]
and the bounded multiplication operator $B_p$ given by $(B_pf)(\xi)=Mf(\xi)$ for $p=1,2$. 
The \emph{numerical abscissa} of a bounded operator $B$ is denoted by 
\[
    \omega(B)
    :=
    \sup\{\operatorname{Re}z:z\in W(B)\}.
\]
We first consider $B_1$ on $X_1$. Let
$f\in X_1$ satisfy $\|f\|_{X_1}=1$, and let
$f^*\in J_{X_1}(f)$. 
Identifying
\[
X_1^*
 =
L^\infty\bigl(\Omega;(\mathbb C^2,\|\cdot\|_\infty)\bigr),
\]
we may represent $f^*$ by a function $g=(g_1,g_2)$ satisfying
\[
\operatorname*{ess\,sup}_{\xi\in\Omega}
\max\{|g_1(\xi)|,|g_2(\xi)|\}=1
\]
and
\[
f^*(h)
 =
\int_0^1
\left(
h_1(\xi)\overline{g_1(\xi)}
+
h_2(\xi)\overline{g_2(\xi)}
\right)\,\mathrm d\xi.
\]

Since $f^*\in J_{X_1}(f)$, equality holds in the duality
estimate. Consequently,
\[
    \overline{g_j(\xi)}f_j(\xi)=|f_j(\xi)|,
    \qquad j=1,2,
\]
for almost every $\xi$ on the set where $f_j(\xi)\neq0$. It follows that
\begin{align*}
    \operatorname{Re}f^*(B_1f)
    &=
    \operatorname{Re}
    \int_0^1
    \left(
        -f_1\overline{g_1}
        +9f_2\overline{g_1}
        -10f_2\overline{g_2}
    \right)
    \,\mathrm{d}\xi
    \\
    &\le
    \int_0^1
    \left(
        -|f_1|
        +9|f_2|
        -10|f_2|
    \right)
    \,\mathrm{d}\xi
    \\
    &=
    -\int_0^1
    \left(
        |f_1|+|f_2|
    \right)
    \,\mathrm{d}\xi
    \\
    &=-1.
\end{align*}
Equality is attained, for example, by choosing
\[
    f=
    \begin{pmatrix}
        \varphi\\0
    \end{pmatrix},
    \qquad
    \|\varphi\|_{L^1}=1,
\]
together with a corresponding norming functional. Therefore, $\omega(B_1)=-1$.

The generated semigroup can be calculated explicitly. Since
\[
    e^{tM}
    =
    \begin{pmatrix}
        e^{-t}&e^{-t}-e^{-10t}\\
        0&e^{-10t}
    \end{pmatrix},
\]
we have
\[
    (e^{tB_1}f)(\xi)
    =
    e^{tM}f(\xi).
\]
All entries of $e^{tM}$ are nonnegative, and both column sums
are equal to $e^{-t}$. Hence, the operator norm from $\mathbb{C}^2$ with $\|\cdot\|_1$  is $\|e^{tM}\|=e^{-t}$.
Consequently, $\|e^{tB_1}\|
        =e^{-t}$, for all $t\geq0$. 
Thus, the negative numerical abscissa gives an exact exponential
decay estimate with prefactor one.

We now consider the same multiplication operator on the Hilbert space $X_2$. Its numerical abscissa is determined by the Hermitian
part of $M$:
\[
    \omega(B_2)
    =
    \lambda_{\max}
    \left(
        \frac{M+M^*}{2}
    \right)=\lambda_{\max}\left(\left(\begin{smallmatrix}
        -1&\frac92\\
        \frac92&-10
    \end{smallmatrix}\right)\right)=\frac{9\sqrt{2}-11}{2}
    >0.
\]
Indeed, if $v\in\mathbb{C}^2$ is a unit eigenvector associated
with the largest eigenvalue of $(M+M^*)/2$ and $f_0(\xi)=v$,  $\xi\in(0,1)$,
then $\|f_0\|_{X_2}=1$ and
\[
    \left.
    \frac{d}{dt}
    \|e^{tB_2}f_0\|_{X_2}^2
    \right|_{t=0}
    =
    2\operatorname{Re}
    \langle B_2f_0,f_0\rangle
    =
    2\omega(B_2)
    >0.
\]
Therefore, the $L^2$-norm increases initially in a suitable
direction. In particular, there is no $\alpha\ge0$ such that
\[
    \|e^{tB_2}\|
    \le e^{-\alpha t},
    \qquad t\ge0.
\]

Nevertheless, $B_2$ is exponentially stable. Indeed,
\[
    e^{tM}
    =
    e^{-t}
    \begin{pmatrix}
        1&1-e^{-9t}\\
        0&e^{-9t}
    \end{pmatrix},
\]
and hence the operator norm in $\mathbb{C}^2$ with the Euclidean norm is 
\[
    \|e^{tM}\|
    \le
    \|e^{tM}\|_{\mathrm{F}}
    \le
    \sqrt{2}\,e^{-t},
\]
where $\|\cdot\|_F$ denotes the Frobenius norm. Thus,
\[
    \|e^{tB_2}\|
    \le
    \sqrt{2}\,e^{-t}.
\]
Hence, the $L^1$ numerical range directly yields strict
exponential decay with prefactor one, whereas the $L^2$
numerical range detects transient growth.

\subsection*{Acknowledgment}
The authors acknowledge the support from the Deutsche Forschungsgemeinschaft DFG (Project number 567253746). 

\bibliographystyle{abbrv}

\bibliography{sample}

\end{document}